\documentclass[12pt]{amsart}
\usepackage[margin=1.2in]{geometry}
\usepackage{color}
\usepackage[sorting = none, backend = bibtex, style=numeric-comp]{biblatex}
\title[Dynamical sampling for source recovery]{Characterization of frames for source  recovery from dynamical samples}

\author[Akram Aldroubi, Rocio Diaz Martin, Le Gong, Javad Mashreghi,  Ivan Medri]{A. Aldroubi, R. Diaz Martin, L. Gong, J. Mashreghi, I. Medri}

\usepackage{hyperref}

\date{\vspace{-5ex}}
\usepackage{amsthm, amsfonts, mathrsfs}
\usepackage{ dsfont }
\usepackage{amssymb}
\usepackage{enumerate}
\usepackage{graphicx}
\usepackage{subfig}

\usepackage{float}
\usepackage{bbm}
\usepackage{amsmath}
\usepackage{comment}
\usepackage{hyperref}
\usepackage{listings}
\usepackage{color}

\usepackage{diagbox}
\usepackage[dvipsnames]{xcolor}
\usepackage{algorithm}
\usepackage{booktabs}
\usepackage{algorithmicx}
\usepackage[noend]{algpseudocode}
\usepackage{tikz}
\allowdisplaybreaks

\hypersetup{
    colorlinks,
    citecolor=blue,
    filecolor=blue,
    linkcolor=blue,
    urlcolor=blue
}
\numberwithin{equation}{section}
\newtheorem{theorem}[equation]{Theorem}
\newtheorem{proposition}[equation]{Proposition}
\newtheorem{lemma}[equation]{Lemma}

\newtheorem{corollary}[equation]{Corollary}
\newtheorem{definition}[equation]{Definition}

\newcommand{\HH}{\mathcal{H}}

\newcommand{\C}{\mathbb{C}}

\newcommand{\N}{\mathbb{N}}

\usepackage[foot]{amsaddr}

\definecolor{aacolor}{rgb}{0.05, 0.75, 1}

\begin{document}
	\date{}
 \thanks{The research of Akram Aldroubi and Le Gong is supported in part by grant DMS-2208030. Javad Mashreghi was supported by grants from the Fulbright Foundation and the Canada Research Chair program.}
\address{\textrm{(Akram Aldroubi)}
Department of Mathematics,
Vanderbilt University,
Nashville, Tennessee 37240-0001 USA}
\email{aldroubi@math.vanderbilt.edu}

\address{\textrm{(Roc\'io Mart\'in D\'iaz)}
	Department of Mathematics,
	Vanderbilt University,
	Nashville, Tennessee 37240-0001 USA}
\email{rocio.p.diaz.martin@vanderbilt.edu}

\address{\textrm{(Le Gong)}
	Department of Mathematics,
	Vanderbilt University,
	Nashville, Tennessee 37240-0001 USA}
\email{le.gong@vanderbilt.edu}

\address{\textrm{(Javad Mashreghi)}
        Laval University,
        Qu\'ebec, QC, G1V 0A6,  Canada}
\email{javad.mashreghi@mat.ulaval.ca}

\address{\textrm{(Ivan Medri)}
        Tennessee State University, Department of Computer Science,
        Nashville, TN 37209, USA}
\email{imedri@tnstate.edu}


\maketitle

\begin{abstract}
In this paper, we address the problem of recovering constant source terms in a discrete dynamical system represented by $x_{n+1} = Ax_n + w$, where $x_n$ is the $n$-th state in a Hilbert space $\mathcal{H}$, $A$ is a bounded linear operator in $\mathcal{B}(\mathcal{H})$, and $w$ is a source term within a closed subspace $W$ of $\HH$. Our focus is on the stable recovery of $w$ using time-space sample measurements formed by inner products with vectors from a Bessel system $\mathcal{G} \subset \mathcal{H}$. We establish the necessary and sufficient conditions for the recovery of $w$ from these measurements, independent of the unknown initial state $x_0$ and for any $w \in W$. This research is particularly relevant to applications such as environmental monitoring, where precise source identification is critical.
\end{abstract}

\section{Introduction}

\subsection{Dynamical sampling for source recovery}

The problem we study in this paper is the recovery of constant source terms driving a discrete dynamical system, using time-space samples of an evolving physical quantity. Specifically, we consider the following discrete-time dynamical system:
\begin{equation}\label{model}
x_{n+1} = Ax_n + w, \quad n \in \mathbb{N}, \qquad w \in W,
\end{equation}
where $x_n \in \mathcal{H}$ is the \textit{$n$-th state of the system}, and $\mathcal{H}$ is a separable Hilbert space. The operator $A \in \mathcal{B}(\mathcal{H})$ is called the \textit{dynamic operator}, $w\in W\subseteq \mathcal{H}$ is the \textit{source} or \textit{forcing term}, and $W$ is a closed subspace of $\mathcal{H}$. The term $x_0 \in \mathcal{H}$ is called the \textit{initial state}. Time-space sample measurements
\begin{equation} \label{DatMat}
\mathcal{D}(x_0, w) = \left[\langle x_n, g_j\rangle\right]_{n,j}
\end{equation}
are obtained by inner products $\langle x_n, g_j\rangle$ with vectors of a Bessel system $\mathcal{G} = \{g_j\}_{j \geq 1} \subset \mathcal{H}$, referred to as the set of \textit{spatial sampling vectors}, and organized in the matrix $\mathcal{D}(x_0, w)$. The data matrix is called the \textit{data} of the system (also called the set of \textit {time-space samples}, \textit{measurements}, or \textit{observations}). 
The problem we will analyze is to find necessary and sufficient conditions for recovering $w$ from the data \eqref{DatMat} in a \textit{stable way}, independent of the unknown $x_0$ and for any $w \in W$. The concept of stability will be made precise in Section \ref{S:proof-forsubspace-tech}.

For example, if $\HH=\ell^2$, the value $x_n(j)$ represents the value of the state at time $n$ and \textit{spatial position} $j$. Given an orthonormal basis $\{b_k\}_{k \in K}$ for $W$, the vector $w = \sum_k c_kb_k$ can be viewed as a weighted sum of source terms $b_k$, each located at positions $k \in K$, with magnitude $c_k$. A set of spatial sampling vectors $\mathcal G = \{g_j\}_{j\geq 1}$,  consisting of spatial sampling vectors (not necessarily in $W$), can be used to obtain the space time-samples \eqref{DatMat}. 

The mathematical problem described above is inspired by environmental monitoring applications for identifying the locations and magnitude of pollution sources. Typically, this necessitates the strategic placement of sensors across different locations to collect relevant data. For instance, in the topic of air pollution highlighted in \cite{RDCV12} and related works, the goal is to determine the emission levels from a specific number of smokestack sources by employing a limited set of sensors.

The recent work \cite{aldroubi2023dynamical}, provides necessary and sufficient conditions for recovering a stationary source in a finite-dimensional dynamical system. In this article, the underlying space is an infinite-dimensional separable Hilbert space $\HH$, and the techniques and some of the results are vastly different.

\subsection{Context of this work}
The area of \textit{Dynamical Sampling} has been extensively studied in the literature \cite{ APT15,ACCMP17,  ACMT17,  ADK13, ADK15, ACCP23, AGHJKR21,  AP17, CMPP17, CMPP20,  CJS15, SCCM23, UZ21, ZLL17}.  Dynamical sampling problems are connected to several areas of mathematics, including control theory, frame theory,  functional analysis, and harmonic analysis \cite{AKh17, ACCP21, ashbrock2023dynamical, BH23, BK23, cabrelli2021multi, cabrelli2022frames,   CH19, CH23, MMM21, DMM21, martin2023error, FS19, GRUV15, KS19, Men22, MT23}, as well as having numerous applications in science and engineering \cite{RBD21, AD20, MBD15, MD17}. The three main problems in the area of dynamical sampling are as follows.
\begin{enumerate}[(i)]
\item \textit{System identification}: Recovering the dynamic operator $A$ by knowledge of the set of measurements $\mathcal{D}(x_0,w)$. See, e.g., \cite{AHKLLV18, AK16, CT22, Tan17}.

\item \textit{Initial state recovery}: Assuming the $A$ is known and $\omega\equiv 0$ or known, the objective is to determine the necessary and sufficient conditions on the spatial sampling set of vectors $\mathcal{G}=\{g_j\}_{j\geq 1}$ and on the operator $A$ for recovering the initial condition $x_0$ in a stable way from the data $\mathcal{D}(x_0,w)$. This problem is also known as the ``time-space trade-off in sampling for recovering the initial state'' since the goal is to exploit time to achieve reconstruction even when $\mathcal{G}$ is finite. See, e.g., \cite{ACCMP17, ACMT17, CMPP17, CMPP20, cabrelli2021multi,  DMM21, MMM21, ZLL17, Zlo22}. 
   
\item \textit{Source recovery}: The focus is on identifying specific types of source terms that drive the dynamical system. See, e.g., \cite{aldroubi2023dynamical, AGK23, AHKK23}.
\end{enumerate}

In this work, we investigate the previously mentioned source recovery problem within the dynamical sampling framework. The key distinction, compared to those in \cite{aldroubi2023dynamical, AGK23, AHKK23}, lies in establishing the necessary and sufficient conditions under which the set $\mathcal{G} = \{g_j\}_{j\geq 1}$ enables stable recovery of the sources. Conversely, the cited works \cite{aldroubi2023dynamical, AGK23, AHKK23} prescribe a set $\mathcal{G}$ and provide algorithms for approximating specific types of time-dependent source terms.

\subsection{Organization of the paper} The organization of the paper is as follows. In Section \ref{S:math-description} we describe the intrinsic mathematical structure of the dynamical system that is our main concern in this work. Section \ref{S:main-results} contains all the main results, consisting of {\em five major theorems}. This part is followed by three consecutive Sections \ref{S:proof-forsubspace-tech}, \ref{S:proofs-part1}, and \ref{S:proofs-part2} which contain the proofs of Main Theorems. In fact, Section \ref{S:proof-forsubspace-tech} contains some technical lemmas that are interesting in their own right, as well as a detailed description of $\mathcal{B}(\ell^2,\ell^\infty)$ and $\mathcal{B}^s(\ell^2,\ell^\infty)$ spaces. The latter is crucial since it is used as the ambient space of measurements in the dynamical system and is exploited in the stable recovery of source file. Finally, Section \ref{S:Example} contains a delicate descriptive example showing that even if the source $w$ belongs to a one-dimensional subspace $W$ within $\HH$, the recovery of $w$ requires an infinite number of time samples when $\HH$ is an infinite-dimensional space. This example reveals the sharpness of our main results. Section \ref{S:concluding} constitutes the final segment of the paper, encompassing concluding remarks, potential generalizations, and avenues for future research.

\section{The mathematical description of problem} \label{S:math-description}
As part of the main problem, given a dynamical system \eqref{model}, we wish to recover the source term $w$ in a stable way from the data provided in measurements $\mathcal D (x_0,w)$. To describe the notion of {\em stable reconstruction} we need to describe some ambient spaces $\mathcal{B}$ in which the data sits together with an appropriate norm $\|\cdot\|_{\mathcal{B}}$ in each case. This setting allows us to describe the reconstruction operator $\mathcal R$ as a continuous linear mapping from the data space $\mathcal{B}$ to the Hilbert space $\HH$ containing the source term $w$. 

\subsection{The measurement space}
There are two cases of dynamical systems that we will wish to consider. Briefly speaking, they are as follows.
\begin{enumerate}[(i)]
\item In the first case, the data matrix $\mathcal{ D}(x_0,w)=\left[\langle x_n,g_j\rangle\right]_{n\in [N], \, j\geq 1}$ is obtained from finitely many iterations, where $[N]=\{0, 1,2,\dots, N-1\}$, $N\geq1$.
\item In the second setting, the data matrix $\mathcal{ D}(x_0,w)=\left[\langle x_n,g_j\rangle\right]_{n\ge0, \, j\geq 1}$ stems from infinitely many time iterations.
\end{enumerate}

In the first case, all data measurements sit in the space  $\mathcal{B} (\ell^2,\C^N)$, which can be described as the family of all infinite matrices $D=[d_{ij}]$ with (finitely many) $N$ rows $r_1,\dots, r_N$, where each row $r_i=(d_{i1},d_{i2},\dots)\in \ell^2$. This space $\mathcal{B} (\ell^2,\C^N)$ is endowed with the norm 
\begin{equation} \label{E:Norm-N}
\|D\|_{\ell^2 \to \C^N} 
=\sum_{i=1}^{N}\left( \sum_{j=1}^{\infty} |d_{ij}|^2 \right)^{1/2}, \qquad \text { for } D \in \mathcal{B}(\ell^2,\C^N).    
\end{equation}

For the second case of infinitely many time iterations, we will use the space $\mathcal{B}^s(\ell^2,\ell^\infty)$ which is a closed subspace of $\mathcal{B}(\ell^2,\ell^\infty)$. The latter is the family of all infinite matrices for which the norm 
\begin{equation} \label{E:Norm-infty}
\|D\|_{\ell^2 \to \ell^\infty} = \sup_{i \geq 1} \left( \sum_{j=1}^{\infty} |d_{ij}|^2 \right)^{1/2}, \qquad \text { for } D \in \mathcal{B}(\ell^2,\ell^\infty),
\end{equation}
is finite. The former space $\mathcal{B}^s(\ell^2,\ell^\infty)$ is the closed subspace consisting of matrices whose rows form a Cauchy sequence in $\ell^2$. More explicitly, we provide the following definition.
\begin {definition} The space $\mathcal{B}^s(\ell^2, \ell^\infty)$ is the set of matrices $\{D = [d_{i,j}]: i \ge 1, j \ge 1\}$ such that each row $r_i$ of $D$ belongs to $\ell^2$, and there exists a $t \in \ell^2$ such that $\lim_{i \to \infty} \|r_i - t\|_{\ell^2} = 0$. The norm $\|D\|_{\ell^2\to\ell^\infty}$ is defined as $\sup_{i \ge 1} \|r_i\|_{\ell^2}$.
\end{definition}

Note that due to the equivalence of norms in $\C^N$,we may replace $\sum\limits_{i=1}^{N}$ by $\sup\limits_{1 \leq i \leq N}$  in \eqref{E:Norm-N}, and so $\mathcal{B}^s(\ell^2,\C^N)=\mathcal{B}(\ell^2,\C^N)$. A detailed description of these spaces, in particular,  
an equivalent description of $\mathcal{B}^s(\ell^2, \ell^\infty)$, 
is available in Section \ref{S:proof-forsubspace-tech}. Throughout the general description of the spaces $\mathcal{B}(\ell^2,\C^N)$ and $\mathcal{B}(\ell^2,\ell^\infty)$, we use the index $i$, commencing from the initial value $1$, for the rows of matrices involved in the discussion. However, when analyzing dynamical systems, we adopt a different indexing scheme, mostly denoted by $n$ and starting at $0$.

\subsection{Generalized source recovery problem}
We also treat dynamical systems that are a generalized version of \eqref{model}. In this general setting, we assume that the states $x_n$, $n \geq 1$, are obtained via a recursive equation
\begin{equation} \label{model-2}
x_n = \mathcal{F}_n(x_0,\dots,x_{n-1},w), \qquad n \geq 1,  
\end{equation}
with $w$ belonging to a closed subspace $W$ of $\HH$. In particular, $\mathcal{F}_n$ can be a nonlinear functional of its arguments. Another prototype example of such a general system is
\[
x_n = A_{n,0}x_0+\dots+A_{n,n-1}x_{n-1}+B_{n}w, \qquad n \geq 1,
\]
where $A$s and $B$s are bounded linear operators on $\HH$. 

To present some solid results in the setting \eqref{model-2}, we assume the system satisfies the following properties.
\begin{enumerate}[(i)]
\item For each $w \in W$, there is a corresponding {\em unique stationary state}. More explicitly, given any $w \in W$, there is an initial state $x_0(w)$ such that 
\[
x_n=x_0(w), \qquad n \geq 1. 
\]
\item The correspondence between $w$ and its unique stationary state $x_0(w)$ is bounded. That is, the mapping $\mathcal{S}:W\to\HH$ defined by $\mathcal{S}(w):=x_0(w)$ is a bounded linear operator, and
we call $\mathcal{S}$ the {\em stationary mapping operator}.
\item For any source term $w \in W$ and any arbitrary initial state $x_0 \in \HH$, we have
\[
\lim_{n \to \infty} x_n = \mathcal{S}(w), 
\]
where the above limit is in $\|\cdot\|_\HH$.
\end{enumerate}

As an illustrative example, when considering a dynamical system of the form \eqref{model}, under the hypothesis $\rho(A)<1$, we will see that $x_0(w)=(I-A)^{-1}w$ is the unique stationary state corresponding to a given $w$. Besides, notice that if $\mathcal{F}_n$ is linear and $\mathcal{S}$ is well-defined, i.e., (i) holds, then $\mathcal{S}$ is necessarily a linear mapping.

\begin{definition} \label {genmod}
 A dynamical system  \eqref{model-2} satisfying the above properties (i)--(iii) will be referred to by the quadruple $(\HH, \, W, \mathcal{F}, \, \mathcal{S})$.   
\end{definition}

\subsection{Stable recovery}
Consider a dynamical system of the form \eqref{model} or of the form $(\HH,W,\mathcal{F},\mathcal{S})$, starting at an arbitrary initial state $x_0\in \HH$ with measurements $\mathcal{D}(x_0,w)$ given by sampling through a Bessel sequence $\mathcal{G}=\{g_j\}_{j\geq 1}$  in $\HH$ as in \eqref{DatMat}.
\begin{enumerate}[(i)]
\item If there are finitely many time iterations, we say that the source term $w\in W \subseteq \HH$ can be recovered from the data $\mathcal{D}(x_0,w)$ in a stable way if there exists a bounded linear operator 
$\mathcal{R}: \mathcal{B}(\ell^2,\C^N)\to \HH$ such that 
\[
\mathcal{R} \big( \mathcal{D}(x_0,w) \big) = w 
\]
for all $x_0\in\HH$ and all $w\in W$.

\item If we have infinitely many time iterations, we say that the source term $w\in W \subseteq \HH$ can be recovered from the data $\mathcal{D}(x_0,w)$ in a stable way if there exists a bounded linear operator 
$\mathcal{R}: \mathcal{B}^s(\ell^2,\ell^\infty)\to \HH$ such that 
\[
\mathcal{R} \big( \mathcal{D}(x_0,w) \big) = w 
\]
for all $x_0\in\HH$ and all $w\in W$.
\end{enumerate}

The differences between the measurement spaces $\mathcal{B}(\ell^2,\C^N)$ and $\mathcal{B}^s(\ell^2,\ell^\infty)$, and consequently the emerging reconstruction operators $\mathcal{R}$, are profound and is discussed in depth in the following sections.

\section{Main Results} \label{S:main-results}
In this section, we gather the main results accompanied by concise descriptions. In later sections, we provide the proofs.

\subsection{The reconstruction}
Our first results reveal the main property of the $\mathcal{B}^s(\ell^2, \ell^\infty)$ space and its role in the stable reconstruction process. 


\begin{theorem} \label{T:B-strong-recovery}
Let $\HH$ be a separable Hilbert space, and let $\mathcal{G}=\{g_j\}_{j\geq 1}$ be any Bessel sequence in $\HH$ with optimal Bessel bound $C_{\mathcal{G}}>0$. Then, for each  $D=[d_{ij}] \in \mathcal{B}^{s}(\ell^2,\ell^\infty)$, the limit
\[
\lim D\mathcal{G} :=\lim_{i \to \infty} \sum_{j=1}^{\infty} d_{ij} g_j
\]
exists in $\HH$ and, moreover, the mapping
\[
\begin{array}{cccc}
\mathcal{R}_{\mathcal{G}}: & \mathcal{B}^{s}(\ell^2,\ell^\infty) &  \longrightarrow  & \HH\\
& D & \longmapsto & \lim D\mathcal{G} 
\end{array}
\]
is a well-defined bounded operator whose norm is precisely $\sqrt{C_{\mathcal{G}}}$.
\end{theorem}

\subsection{Finite time iterations}
In this part, the reconstruction of the source term is done by using a finite number of time iterations of the dynamical system \eqref{model}.

In our first result, the source $w$ can be any point of the ambient space $\HH$. From the practical point of view, this case is not as interesting as the upcoming restricted case to closed subspaces, since in general, the sources are not located at every spatial location. However,  from the mathematical point of view, it has an elegant description of the solution to the source recovery problem. 

\begin{theorem}\label{thm1}
Let $\HH$ be a separable Hilbert space, and let $\mathcal{G}=\{g_j\}_{j\geq 1}$ be a Bessel sequence in $\HH$. Consider the dynamical system \eqref{model}, with an arbitrary initial state $x_0\in \HH$. Then the source term $w\in\HH$ can be recovered from the measurements 
$\mathcal{ D}(x_0,w)=\left[\langle x_n,g_j\rangle\right]_{n\in [N], \, j\geq 1} $
in a stable way for some $1\leq N<\infty$ if and only if $\mathcal{G}=\{g_j\}_{j\geq 1}$ is a frame for $\HH$.
\end{theorem}

In the next result, we restrict the source term to be in a closed subspace $W$ of $\HH$. From a practical point of view, despite being mathematically more challenging, this is the most interesting case. As a matter of fact, in applications such as environmental monitoring, one has prior knowledge of where the main pollution sources are located which is translated into considering closed subspaces of the ambient space $\HH$. However, the mathematical description of a solution turns out to be more 
subtle.

\begin{theorem} \label{thm2}
Let $\HH$ be a separable Hilbert space, and let $\mathcal{G}=\{g_j\}_{j\geq 1}$ be a Bessel sequence in $\HH$. Let $W$ be a closed subspace of $\HH$, let $P_W:\HH\to\HH$ be the orthogonal projection onto $W$, and assume that source term $w$ belongs to $W$. Consider the dynamical system \eqref{model}, with an arbitrary initial state $x_0\in \HH$, and with $1\notin\sigma(A)$. If the source term $w\in W$ can be recovered  from the measurements 
$\mathcal{D}(x_0,w)=
\left[\langle x_{n}, g_j\rangle\right]_{n\in [N], \, j\geq 1} $ in a stable way for some $1\leq N<\infty$,
then $\{P_W(I-A^*)^{-1}g_j\}_{j\geq 1}$ is a frame for $W$. 
\end{theorem}

There is an example in Section \ref{S:Example}, that shows that in the case where $W\subsetneq \HH$, the source $w$ cannot be recovered from finitely many iterations $N$. Thus, on the one hand, this example reveals the sharpness of Theorem \ref{thm2}. On the other hand, in the next subsection, we consider the case where $W\subsetneq \HH$ and the data $\mathcal D (x_0,w)$ consists of infinitely many time samples.

\subsection {Infinite time iterations}
Dynamical system \eqref{model} is a special of the more general dynamical system \eqref{model-2}. In fact, the Model  \eqref{model-2} can be even nonlinear. The following theorem provides a full characterization of stable reconstruction from data measurements $\mathcal{D}(x_0,w)$ for the system \eqref{model-2}. 

\begin{theorem}\label{thmforsubspace-general}
Let $\HH$ be a separable Hilbert space, let $W$ be a closed subspace of $\HH$, and let $\mathcal{G}=\{g_j\}_{j\geq 1}$ be a Bessel sequence in $\HH$.
Consider the dynamical system $(\HH, \, W, \mathcal{F}, \, \mathcal{S})$ (see Definition \ref {genmod}) with any initial state $x_0 \in \HH$, and assume that $\mathcal{F}$ is linear. Then, each source term $w\in W$ can be recovered from the measurements $\mathcal{D}(x_0,w) = [\langle x_{n}, g_j\rangle]_{n\ge 0, j\ge 1}$ in a stable way if and only if $\{\mathcal{S}^{*} g_j\}_{j \geq 1}$ is a frame for $W$.
\end{theorem}

As a special case of Theorem \ref {thmforsubspace-general}, we characterize stable reconstruction for the dynamical system \eqref{model} created with an operator $A$ with the spectral radius $\rho (A)<1$.

\begin{theorem}\label{thmforsubspace}
Let $\HH$ be a separable Hilbert space, let $W$ be a closed subspace of $\HH$,  and let $\mathcal{G}=\{g_j\}_{j\geq 1}$ be a Bessel sequence in $\HH$. Consider the dynamical system \eqref{model} with an arbitrary initial state $x_0\in \HH$, and with $\rho(A)<1$. Then each source term $w\in W$ of the system can be recovered from the measurements $\mathcal{D}(x_0,w) = [\langle x_{n}, g_j\rangle]_{n\ge 0, j\ge 1}$ in a stable way if and only if $\{P_W(I-A^*)^{-1}g_j\}_{j \geq 1}$ is a frame for $W$.
\end{theorem}

\section{Stable reconstruction and the proof of Theorem \ref{T:B-strong-recovery}} \label{S:proof-forsubspace-tech}

Before presenting the proof of Main Theorems, we need to develop the required mathematical background. In particular, our main goal is to provide a clear foundation for the notion of {\em stable reconstruction} which was exploited in the announcement of theorems. This important concept needs a detailed discussion, which is fulfilled in this section.
 
\begin{lemma} \label{L:normA-l2linfty}
Let $D$ be a bounded operator from the sequence Hilbert space $\ell^2$ to the sequence Banach algebra $\ell^\infty$. Let $[d_{ij}]$ be the matrix representation of $D$ with respect to the canonical basis $(e_n)_{n \geq 1}$, i.e.,
\[
De_j = \sum_{i=1}^{\infty} d_{ij} e_i, \qquad j \geq 1.
\]
Then
\[
\|D\|_{\ell^2 \to \ell^\infty} = \sup_{i \geq 1} \left( \sum_{j=1}^{\infty} |d_{ij}|^2 \right)^{1/2}.
\]
\end{lemma}

\begin{proof}
Let $x = (x_j)_{j \geq 1} \in \ell^2$. Then, by Cauchy--Schwartz inequality,
\begin{eqnarray*}
\|Dx\|_{\ell^\infty} 
&=& \sup_{i \geq 1} |(Dx)_i|= \sup_{i \geq 1} \left| \sum_{j=1}^{\infty} d_{ij}x_j \right|\leq \sup_{i \geq 1} \left( \sum_{j=1}^{\infty} |d_{ij}|^2 \right)^{1/2} \, \|x\|_{\ell^2}.\
\end{eqnarray*}
Hence,
\[
\|D\|_{\ell^2 \to \ell^\infty} \leq \sup_{i \geq 1} \left( \sum_{j=1}^{\infty} |d_{ij}|^2 \right)^{1/2}.
\]

For the reverse inequality, fix a row $i$ in $D$. Define the vector $x=(x_j)_{j \geq 1}$ by
\[
x_j := \left\{
\begin{array}{ccc}
|d_{ij}|^2/d_{ij} & \mbox{if} & d_{ij} \neq 0,\\
& & \\
0 & \mbox{if} & d_{ij}=0.\\
\end{array} \right.
\]
These coefficients are designed to have
\[
(Dx)_{i} = \sum_{j=1}^{\infty} d_{ij}x_j  = \sum_{j=1}^{\infty} |d_{ij}|^2,
\]
and, at the same time,
\[
\|x\|_{\ell^2} = \left( \sum_{j=1}^{\infty} |d_{ij}|^2 \right)^{1/2}.
\]
Thus,
\[
\frac{\|Dx\|_{\ell^\infty}}{\|x\|_{\ell^2}} 
\geq 
\frac{|(Dx)_{i}|}{\|x\|_{\ell^2}} 
= \left( \sum_{j=1}^{\infty} |d_{ij}|^2 \right)^{1/2},
\]
which gives
\[
\|D\|_{\ell^2 \to \ell^\infty} \geq  \left( \sum_{j=1}^{\infty} |d_{ij}|^2 \right)^{1/2}.
\]
\end{proof}

We denote the family of all bounded linear operators from $\ell^2$ to $\ell^\infty$ by $\mathcal{B}(\ell^2,\ell^\infty)$. Lemma \ref{L:normA-l2linfty} characterizes such operators as those whose rows, in the canonical matrix representation, are uniformly bounded in $\ell^2$. We define a closed subspace of $\mathcal{B}(\ell^2,\ell^\infty)$ and use it to tackle the question of stable reconstruction.  The precise definition is as follows. 

Let $\HH$ be a separable Hilbert space. Then $\mathcal{B}^{s}(\ell^2,\ell^\infty)$ (where $s$ stands for strong) consists of all operators $D \in \mathcal{B}(\ell^2,\ell^\infty)$ such that the limit
\begin{equation} \label{E:row-A-G}
\lim_{i \to \infty} \sum_{j=1}^{\infty} d_{ij} g_j    
\end{equation}
exists in the norm of $\HH$ for any Bessel sequence $\mathcal{G}=\{g_j\}_{j\geq 1} \subset \HH$. By the unitary equivalence of separable Hilbert spaces, the definition does not depend on $\HH$. Also note that since $\mathcal{G}$ is a Bessel sequence and each row of $D$ is in $\ell^2$, the sum in \eqref{E:row-A-G} is well-defined and represents an element of $\HH$. In the following, we will write $\lim D\mathcal{G}$ for the limit in \eqref{E:row-A-G}. More explicitly, whenever $D \in \mathcal{B}^{s}(\ell^2,\ell^\infty)$ and $\mathcal{G}$ is a Bessel sequence in $\HH$, we will write
\begin{equation} \label{E:row-A-G-2}
\lim D\mathcal{G} := \lim_{i \to \infty} \sum_{j=1}^{\infty} d_{ij} g_j   
\end{equation}

\begin{lemma} \label{L:B-strong}
Let $D = [d_{ij}] \in \mathcal{B}(\ell^2,\ell^\infty)$. Then $D \in \mathcal{B}^{s}(\ell^2,\ell^\infty)$ if and only if
the rows of $D$ are norm convergent in $\ell^2$, i.e., there is a vector $t \in \ell^2$ such that
\[
\lim_{i \to \infty} \|r_i - t\|_{\ell^2} = 0, 
\]
where $r_i := (d_{i1}, d_{i2},\dots)$ denotes the $i$-th row of $D$.
\end{lemma}

\begin{proof}
Assume that $D \in \mathcal{B}^{s}(\ell^2,\ell^\infty)$. If we consider even a single sequence $\mathcal{G}$ which is an orthonormal basis in $\HH$, then the series in \eqref{E:row-A-G} is unitarily equivalent to the vector $r_i := (d_{i1}, d_{i2},\dots) \in \ell^2$. Thus, the assumption on the existence of a limit precisely means that the rows are convergent in $\ell^2$. 

Conversely, let $D \in \mathcal{B}(\ell^2,\ell^\infty)$, assume that its rows are convergent in $\ell^2$ norm to a vector $t=(t_1,t_2,\dots)\in\ell^2$, and fix an arbitrary Bessel sequence $\mathcal{G}=\{g_j\}_{j\geq 1}\subset\HH$. Then,
\[
h := \sum_{j=1}^{\infty} t_{j} g_j 
\]
is a well-defined element of $\HH$. Moreover,
\begin{eqnarray*}
\left\| h - \sum_{j=1}^{\infty} d_{ij} g_j \right\|_{\HH} 
&=& \left\| \sum_{j=1}^{\infty} t_{j} g_j - \sum_{j=1}^{\infty} d_{ij} g_j \right\|_{\HH}\\
&=& \left\| \sum_{j=1}^{\infty} (t_{j} - d_{ij}) g_j \right\|_{\HH}\\
&\leq& \sqrt{C_{\mathcal{G}}} \, \|t-r_i\|_{\ell^2},
\end{eqnarray*}
where $C_{\mathcal{G}}$ is the optimal bound of the Bessel sequence $\mathcal{G}$ (see, for e.g. \cite[Thm. 3.2.3]{christensen2003introduction}). By assumption, $\|t-r_i\|_{\ell^2} \to 0$, as $i \to \infty$. Hence, 
\[
\lim_{i \to \infty} \sum_{j=1}^{\infty} d_{ij} g_j = h,
\]
which means that $D \in \mathcal{B}^{s}(\ell^2,\ell^\infty)$.
\end{proof}

\begin{corollary} \label{C:B-strong}
$\mathcal{B}^{s}(\ell^2,\ell^\infty)$ is a closed subspace of $\mathcal{B}(\ell^2,\ell^\infty)$.
\end{corollary}

\begin{proof}
Let $D^{(n)}$, $n \geq 1$, be a sequence in $\mathcal{B}^s(\ell^2,\ell^\infty)$ which converges to $D \in \mathcal{B}(\ell^2,\ell^\infty)$ in the topology of $\mathcal{B}(\ell^2,\ell^\infty)$. Denote the $i$-th rows of $D^{(n)}$ and $D$ respectively by $r_{i}^{(n)}$ and $r_{i}$. Write
\[
r_{i} - r_{i'} =  (r_{i}-r_{i}^{(n)}) +  (r_{i'}^{(n)}-r_{i'}) + (r_{i}^{(n)}-r_{i'}^{(n)}).
\]
For the first two terms on the right side we have
\[
\|r_{i}-r_{i}^{(n)}\|_{\ell^2}  \leq \|D-D^{(n)}\|_{\ell^2 \to \ell^\infty},
\]
and 
\[
\|r_{i'}-r_{i'}^{(n)}\|_{\ell^2}  \leq \|D-D^{(n)}\|_{\ell^2 \to \ell^\infty},
\]
Therefore, given $\varepsilon>0$, we take $n$ so large that $\|D-D^{(n)}\|_{\ell^2 \to \ell^\infty} \leq \varepsilon$. Hence,
\[
\|r_{i} - r_{i'}\|_{\ell^2} \leq  2\varepsilon + \|r_{i}^{(n)}-r_{i'}^{(n)}\|_{\ell^2}.
\]
Now, by Lemma \ref{L:B-strong} applied to $D^{(n)}$, there is and index $i_0$ such that
\[
\|r_{i}^{(n)}-r_{i'}^{(n)}\|_{\ell^2} < \varepsilon, \qquad i,i' > i_0.
\]
Therefore, for all $i,i' > i_0$, we will have
\[
\|r_{i} - r_{i'}\|_{\ell^2} \leq 3 \varepsilon. 
\]
This means that the rows of $D$ are norm Cauchy, and thus norm convergent. Hence, one again using Lemma \ref{L:B-strong}, we conclude that $D \in \mathcal{B}^{s}(\ell^2,\ell^\infty)$.
\end{proof}

All the previous results were designed to arrive at Theorem \ref{T:B-strong-recovery} which is a fundamental result. We are now able to prove this theorem.

\subsection{Proof of Theorem \ref{T:B-strong-recovery}}
Fix any $D \in \mathcal{B}^{s}(\ell^2,\ell^\infty)$. Let $r_i := (d_{i1}, d_{i2},\dots)$ denotes the $i$-th row of $D$. Define
\[
(D\mathcal{G})_{i} := \sum_{j=1}^{\infty} d_{ij} g_j, \qquad i \geq 1.
\]
Then,
\[
\|(D\mathcal{G})_{i}\|_{\HH} 
= \left\|\sum_{j=1}^{\infty} d_{ij} g_j\right\|_{\HH} 
\leq \sqrt{C_{\mathcal{G}}} \|r_i\|_{\ell^2}
\leq \sqrt{C_{\mathcal{G}}} \|D\|_{\ell^2 \to \ell^\infty}, \qquad i \geq 1.
\]
Let $i \to \infty$ to obtain
\[
\|\mathcal{R}_{\mathcal{G}}(D)\|_{\HH} = \|\lim D\mathcal{G}\|_{\HH} \leq \sqrt{C_{\mathcal{G}}} \|D\|_{\ell^2 \to \ell^\infty}.
\]
Therefore, the operator $\mathcal{R}_{\mathcal{G}}$ is bounded and
\[
\|\mathcal{R}_{\mathcal{G}}\|_{\mathcal{B}^{s}(\ell^2,\ell^\infty) \to \HH} \leq \sqrt{C_{\mathcal{G}}}.
\]

To prove the reverse inequality, let $D$ be the operator in $\mathcal{B}^{s}(\ell^2,\ell^\infty)$ whose rows are all equal to a fixed vector
\[
d = (d_1,d_2,\dots) \in \ell^2.
\]
Then clearly
\[
\|D\|_{\ell^2 \to \ell^\infty} = \|d\|_{\ell^2}, 
\]
and
\[
\mathcal{R}_{\mathcal{G}}(D) = \lim D\mathcal{G} = \sum_{j=1}^{\infty} d_{j} g_{j}.
\]
Hence, the inequality
\[
\|\mathcal{R}_{\mathcal{G}}(D)\|_{\HH} \leq \|\mathcal{R}_{\mathcal{G}}\|_{\mathcal{B}^{s}(\ell^2,\ell^\infty) \to \HH} \, \|D\|_{\ell^2 \to \ell^\infty}
\]
transforms to
\[
\left\| \sum_{j=1}^{\infty} d_{j} g_{j} \right\|_{\HH} \leq \|\mathcal{R}_{\mathcal{G}}\|_{\mathcal{B}^{s}(\ell^2,\ell^\infty) \to \HH} \, \|d\|_{\ell^2}, \qquad d \in \ell^2.
\]
But, since $d \in \ell^2$ is arbitrary, this estimation implies
\[
\sqrt{C_{\mathcal{G}}} \leq \|\mathcal{R}_{\mathcal{G}}\|_{\mathcal{B}^{s}(\ell^2,\ell^\infty) \to \HH}.
\]
\qed

\section{Proofs of theorems for finite time iterations} \label{S:proofs-part1}

\begin{lemma}
    Consider the dynamical system \eqref{model} with any initial state $x_0 \in \HH$. Given a Bessel sequence $\mathcal{G}=\{g_j\}_{j\geq 1}\subset\HH$ and $1\leq N<\infty$, then the data matrix $\mathcal{D}(x_0,w) = [\langle x_{n}, g_j\rangle]_{n\in [N], \, j\geq 1}$ belongs to $\mathcal{B}(\ell^2,\C^N)$.
\end{lemma}
\begin{proof}
The proof holds since $\mathcal{G}=\{g_j\}_{j\geq 1}\subset\HH$ is a Bessel sequence. Indeed, if $C_\mathcal{G}$ is the optimal Bessel bound, then
\begin{align*}
    \|\mathcal{D}(x_0,w)\|_{\ell^2\to\C^N}=\sum_{n=0}^{N-1}\sum_{j=1}^\infty|\langle x_n,g_j\rangle|^2\leq C_{\mathcal{G}}\sum_{n=0}^{N-1}\|x_n\|_{\HH}^2<\infty.
\end{align*}
\end{proof}

\begin{proposition}
Consider the dynamical system \eqref{model}, and let $1\leq N<\infty$. Given a frame $\mathcal{G}=\{g_j\}_{j\geq 1}$ for $\HH$, consider the {data operator} $\mathcal{D}$ defined by 
    \begin{gather*}
       \mathcal{D}: \HH\times \HH\longrightarrow\mathcal{B}(\ell^2,\C^N)\notag \\
       \mathcal{D}(x_0,w):=[\langle x_n,g_j\rangle]_{n\in [N], j\geq 1}.
    \end{gather*}
Then the image of $\mathcal{D}$ is a closed subspace in $\mathcal{B}(\ell^2,\C^N)$.
\end{proposition}

\begin{proof}
Consider a sequence $\{(x_0^k,w^k)\}_{k\in\N}$ in $\HH\times \HH$ such that the sequence $\{\mathcal{D}(x_0^k,w^k)\}_{k\in \N}$ converges in $\mathcal{B}(\ell^2,\C^N)$. Let $r\in \mathcal{B}(\ell^2,\C^N)$ be such limit 
$$
\lim_{k\to\infty}\mathcal{D}(x_0^k,w^k)=r.
$$
We recall that, by definition, $r$ is of the form  $r:=(r_0,r_1,\dots, r_{N-1})$ where, for each $0\leq n\leq N-1$,  $r_n:=(r_{n1},r_{n2},\dots)\in\ell^2$.
    
Since $\mathcal{G}=\{g_j\}_{j\geq 1}$ is a frame for $\HH$, the image of the analysis operator   
\begin{gather*}
\mathcal{A}_{\mathcal{G}}:\HH\longrightarrow\ell^2\\
\mathcal{A}_{\mathcal{G}}(h):=(\langle h,g_j\rangle)_{j\geq 1}
\end{gather*}
is a closed linear subspace in $\ell^2$. Notice that 
\begin{equation} \label{E:tmp-ro-limak}
r_0=\lim_{k\to\infty}\mathcal{A}_{\mathcal{G}}(x_0^k)
\end{equation}
where the above limit is in $\ell^2$-norm. Since the image of $\mathcal{A}_\mathcal{G}$ is closed,  there exists $x_0\in \HH$ such that
\[
r_0=\mathcal{A}_\mathcal{G}(x_0)= \big( \langle x_0, g_1 \rangle, \langle x_0, g_2 \rangle,\dots \big).
\]
Since $\mathcal{G}$ is a frame with optimal lower and upper bounds $c_{\mathcal{G}}$ and $C_\mathcal{G}$, we obtain that
\[
\|x_0^k-x_0\|_\HH^2 
\leq \frac{1}{c_{\mathcal{G}}}\sum_{j=1}^\infty|\langle x_0^k-x_0, g_j\rangle|^2
= \frac{1}{c_{\mathcal{G}}}\|\mathcal{A}_\mathcal{G}(x_0^k)-\underbrace{\mathcal{A}_\mathcal{G}(x_0)}_{r_0}\|_{\ell^2}^2,
\]
and thus, by \eqref{E:tmp-ro-limak},
\begin{equation}\label{eq: conv in H}
\lim_{k\to \infty}\|x_0^k-x_0\|_\HH^2
\leq \lim_{k\to \infty}\frac{1}{c_{\mathcal{G}}}\|\mathcal{A}_\mathcal{G}(x_0^k)-r_0\|_{\ell^2}^2= 0.
\end{equation}
This means that $\{x_0^k\}_{k\in \N}$ converges to $x_0$ in $\HH$. 

Similarly,
\begin{eqnarray*}
r_1=\lim_{k\to\infty}\mathcal{A}_{\mathcal{G}}(Ax_0^k+w^k)=\mathcal{A}_{\mathcal{G}}(Ax_0)+\lim_{k\to\infty}\mathcal{A}_{\mathcal{G}}(w^k).
\end{eqnarray*}
In particular $\{\mathcal{A}_{\mathcal{G}}(w^k)\}_{k\in \N}$ converges in $\ell^2$, and since the image of   $\mathcal{A}_{\mathcal{G}}$ is closed in $\ell^2$, there exists $w\in \HH$ such that
\[
\lim_{k\to \infty}\|\mathcal{A}_{\mathcal{G}}(w^k)-\mathcal{A}_{\mathcal{G}}(w)\|_{\ell^2}^2=0.
\]
By repeating the argument given in \eqref{eq: conv in H} we see that  $\{w^k\}_{k\in \N}$ converges to $w$ in $\HH$. Thus,
\[
r_1=\mathcal{A}_\mathcal{G}(Ax_0+w).
\]
Finally, as the $n$-th state of the dynamical system \eqref{model} with initial state $x_0^k\in \HH$ and source $w^k\in \HH$ can be written as
\[
x_{n}^k = A^{n}x_{0}^k+(I+A+\cdots+A^{n-1})w^k, \qquad n \geq 1,
\]
and, analogously, the 
$n$-th state of the dynamical system \eqref{model} with initial state $x_0\in \HH$ and source $w\in \HH$ can be written as
\[
x_{n} = A^{n}x_{0}+(I+A+\cdots+A^{n-1})w, \qquad n \geq 1,  
\]
we have 
\begin{eqnarray*}
r_n
&=& \lim_{k\to \infty}\mathcal{A}_\mathcal{G}(x_n^k)=\lim_{k\to \infty}\mathcal{A}_\mathcal{G}\left(A^{n}x_{0}^k+(I+A+\cdots+A^{n-1})w^k\right)\\
&=& \mathcal{A}_\mathcal{G}\left(A^{n}x_{0}+(I+A+\cdots+A^{n-1})w\right)=  \mathcal{A}_\mathcal{G}(x_n).
\end{eqnarray*}
This implies that 
\[
\lim_{k\to\infty}\mathcal{D}(x_0^k,w^k)=r=\mathcal{D}(x_0,w),
\]
that is,  $r$ belongs to the image of the data operator $\mathcal{D}$.
\end{proof}

\subsection{Proof of Theorem \ref{thm1}} 
Consider the dynamical system \eqref{model}, and suppose that a stable recovery is possible in $N$ time of iterations. More explicitly, for any $x_0\in\HH$ and any $w\in\HH$, the source $w$ can be recovered by applying a bounded linear operator $\mathcal{R}:\mathcal{B}(\ell^2,\C^N)\rightarrow \HH$ to the measurements $\mathcal{ D}(x_0,w)=\left[\langle x_n,g_j\rangle\right]_{n\in [N], \, j\geq 1}$, where $\mathcal{G}=\{g_j\}_{j\geq 1}$ is a Bessel sequence in $\HH$. Hence, there exists a positive constant  $C$ such that
\begin{equation}\label{bdd}
\left\| \mathcal{R}(D) \right\|_{\HH}^2 \le C \, \sum_{n=0}^{N-1}\sum\limits_{j= 1}^\infty|d_{nj}|^2, \qquad D\in \mathcal{B}(\ell^2,\ell^\infty),
\end{equation} 
 and 
\begin{equation} \label{bdd-2}
\mathcal{R}(\mathcal{D}(x_0,w))=w, \qquad  x_0, w\in \HH.
\end{equation}
Our objective is to show that $\mathcal{G}=\{g_j\}_{j\geq 1}$ is a frame for $\HH$.

According to \eqref{bdd} and \eqref{bdd-2}, we have
\begin{eqnarray}\label{eq: w lower bound}
\|w\|_\HH^2 
&\leq& C \left(\sum_{j=1}^\infty |\langle x_{0}, g_j\rangle|^2+\cdots+\sum_{j=1}^\infty |\langle x_{N-1}, g_j\rangle|^2 \right). 
\end{eqnarray}
Hence, by {\eqref{model}}, we obtain
\begin{equation}\label{bddforw}
\|w\|_\HH^2 \leq C \left(\sum_{j=1}^\infty\sum_{n=0}^{N-1} \left| \left\langle A^nx_0+\sum_{k=0}^{n-1}A^{k} w, g_j \right\rangle \right|^2 \right), 
\end{equation}
where if $n=0$, we understand $A^nx_0+\sum_{k=0}^{n-1}A^{k}$ simply as $x_0$.

Next, for $x_0\in \HH$, we choose $w=(I-A)x_0$, and we substitute it into \eqref{bddforw}. After some simplifications, the above relation is rewritten as 
\[
\|(I-A)x_0\|_\HH^2 \le NC \sum_{j=1}^\infty |\langle x_0, g_j\rangle|^2, \qquad  x_0\in\HH.
\]
Repeated use of this estimation implies
\begin{eqnarray}
\|(I-A)^{k} x_0\|_\HH^2 
&\leq& NC \sum_{j=1}^\infty |\langle (I-A)^{k-1}x_0, g_j\rangle|^2 \notag\\
&\leq& NCC_\mathcal{G}\|(I-A)^{k-1}x_0\|_\HH^2 \notag\\
&\leq& N^2C{^2}C_\mathcal{G}\sum_{j=1}^\infty |\langle (I-A)^{k-2}x_0, g_j\rangle|^2 \notag\\
&\vdots& \notag\\
&\leq& C_k \sum_{j=1}^\infty |\langle x_0, g_j\rangle|^2, \qquad\qquad x_0 \in \HH, \label{IAbound}
\end{eqnarray}
where we have used the fact that $\mathcal{G}= \{g_j\}_{j\geq 1}$ is a Bessel sequence in $\HH$ with optimal bound $C_{\mathcal{G}}$, and where the constant $C_k =N^kC{^k}C_\mathcal{G}^{k-1}$ depends on constant $C$, on the number of time iterations of the dynamical system $N$, and on the power $k\geq 1$.

We go back to \eqref{bddforw} again, but this time we consider $x_0=0$. If we write $A$ as $A=I-(I-A)$, then the expression $\sum\limits_{k=0}^{n-1} A^{k} w$ can be rewritten as 
\[
\sum_{k=0}^{n-1} A^{k} w = \sum_{k=0}^{n-1} \alpha_{n,k} (I-A)^{k} w, \qquad 1 \leq n \leq N,
\]
for some complex coefficients $\alpha_{n,k}$ whose precise values are not relevant here. Thus, we get
\begin{eqnarray*}
\|w\|_\HH^2 
&\leq& C \left(\sum_{j=1}^\infty\sum_{n=0}^{N-1} \left| \left\langle \underbrace{A^nx_0}_{0}+\sum_{k=0}^{n-1}A^{k} w, g_j \right\rangle \right|^2 \right)\\
&=& C \left(\sum_{j=1}^\infty\sum_{n=1}^{N-1} \left| \left\langle \sum_{k=0}^{n-1}A^{k} w, g_j \right\rangle \right|^2 \right)\\
&=& C \left(\sum_{j=1}^\infty \sum_{n=1}^{N-1} \left| \left\langle \sum_{k=0}^{n-1} \alpha_{n,k} (I-A)^{k} w,g_j \right\rangle \right|^2 \right) \\
&\leq& C(\sup|\alpha_{n,k}|)^2\left(\sum_{j=1}^\infty \sum_{n=1}^{N-1}\sum_{k=0}^{n-1}\left|\left\langle (I-A)^{k} w,g_j \right\rangle \right|^2 \right)\\
&\leq&C(\sup|\alpha_{n,k}|)^2(N-1)\left(\sum_{j=1}^\infty \sum_{k=0}^{N-2}\left|\left\langle (I-A)^{k} w,g_j \right\rangle \right|^2 \right)\\
&=& C'\left(\sum_{j=1}^\infty \sum_{k=0}^{N-2} |\langle(I-A)^{k} w,g_j\rangle|^2 \right) \\
&\leq& C' C_\mathcal{G} \sum_{k=1}^{N-2}\|(I-A)^{k} w\|_\HH^2+C'\left(\sum_{j=1}^\infty |\langle w,g_j\rangle|^2 \right) \\
&\leq& C'' \sum_{j=1}^\infty |\langle w, g_j\rangle|^2, 
\end{eqnarray*}
where
$C'=C {(N-1)}(\sup\{|\alpha_{n,k}|: 0\leq k\leq n-1, \, 1\leq n\leq N-1\})^2$, 
the third inequality again stems from the fact that $\mathcal{G}= \{g_j\}_{j \geq 1}$ is a Bessel sequence in $\HH$, and the last inequality holds because of \eqref{IAbound} with $C''=C'(C_\mathcal{G}\sum\limits_{k=1}^{N-2} C_k+1)$.  
Therefore,
\[
\frac{1}{C''} \|w\|_\HH^2\leq  \sum_{j=1}^\infty |\langle w, g_j\rangle|^2\leq C_\mathcal{G}  \|w\|_\HH^2, \qquad  w\in \HH
\]
and we finally conclude that $\mathcal{G}= \{g_j\}_{j\geq 1}$ is a frame for $\HH$.

Now we prove the converse statement: $w$ can be recovered in a stable way  $\mathcal{G}=\{g_j\}_{j \geq 1}$ is a frame for $\HH$. 

If $\mathcal{G}=\{g_j\}_{j \geq 1}$ is a frame for $\HH$, there exists constants $c_\mathcal{G}, C_\mathcal{G}>0$ such that
\begin{equation}\label{bdofframe}
 c_\mathcal{G}\|h\|^2_{\HH} \leq \sum_{j=1}^\infty |\langle h, g_j\rangle|^2 \leq C_\mathcal{G}\|h\|_{\HH}^2,\qquad h\in\HH.
\end{equation}
Additionally, let $\widetilde{\mathcal{G}}=\{\widetilde{g_i}\}_{i \geq 1}$ be the canonical dual frame of $\mathcal{G}=\{g_j\}_{j\geq 1}$. Then, for each $j\geq 1$,
\begin{equation}\label{eq: coeff aij}
A^*g_j = \sum_{i=1}^\infty a_{ij} g_i,  \qquad \text{ where } \quad a_{ij} =\langle A^*g_j, \widetilde{g_i}\rangle.  
\end{equation}

Given the dynamical system \eqref{model}, we have
\begin{align}
\langle w,g_j\rangle
&= \langle x_{n+1},g_j\rangle-\langle Ax_n,g_j\rangle \notag\\
&=\langle x_{n+1},g_j\rangle-\langle x_n,A^*g_j\rangle \notag\\
&= \langle x_{n+1},g_j\rangle- \left\langle x_n,\sum_{i=1}^\infty a_{ij} g_i \right\rangle \notag\\
&= \langle x_{n+1},g_j\rangle-\sum_{i=1}^\infty \overline{a_{ij}}\langle x_n,g_i\rangle. \label{E:coefficients-wgj}
\end{align}
Notice that \eqref{E:coefficients-wgj} holds for any two consecutive states $x_n$ and $x_{n+1}$.

Since $\mathcal{G}$ is a frame for $\HH$,  $w\in \HH$ can be written as 
\[
w = \sum_{j=1}^\infty \langle w, g_j\rangle \widetilde{g}_j.
\]
By \eqref{E:coefficients-wgj}, every coefficient $\langle w, g_j\rangle$ can be written in terms of the measurements 
$$
\{\langle x{_n},g_j\rangle, \, \langle x{_{n+1}},g_j\rangle \}_{j\geq 1},
$$ 
for all values of $n\geq 0$. Therefore, we have the following reconstruction expression for $w$ in terms of the measurements of the dynamical system \eqref{model}:
\begin{equation}\label{eq: reconst with 2}
w = \sum_{j=1}^\infty \left( {\langle x_{n+1},g_j\rangle}-\sum_{i=1}^\infty \overline{a_{ij}}{\langle x_n,g_i\rangle} \right) \widetilde{g}_j.
\end{equation}
In particular, one can consider the space samples of the first two states of the system $x_0$ and $x_1$ and get
\begin{equation}\label{eq: y0 y1}
w = \sum_{j=1}^\infty \left( \langle x_1,g_j\rangle-\sum_{i=1}^\infty \overline{a_{ij}} \, \langle x_0,g_i\rangle \right) \widetilde{g}_j.
\end{equation}
Moreover, from \eqref{bdofframe} and \eqref{E:coefficients-wgj}, we have 
\begin{eqnarray}
\sum_{j=1}^\infty\left|\sum_{i=1}^\infty \overline{a_{ij}}\langle x_0,g_i\rangle\right|^2
&=&\sum_{j=1}^\infty |\langle Ax_0, g_j\rangle|^2 \notag\\
&\leq& C_\mathcal{G}\|Ax_0\|_\HH^2 
\leq C_\mathcal{G}\|A\|_{\HH\to\HH}^2\|x_0\|_\HH^2 \notag\\
&\leq& \frac{C_\mathcal{G}}{c_\mathcal{G}} \, \|A\|_{\HH\to\HH}^2 \sum_{j=1}^\infty |\langle x_0, g_j\rangle|^2, \label{eq: alg is continuous}
\end{eqnarray}
and thus
\begin{eqnarray}
\left\|\sum_{j=1}^\infty \left(\langle x_{1},g_j\rangle-\sum_{i=1}^\infty \overline{a_{ij}}\langle x_0,g_i\rangle\right)\widetilde{g_j}\right\|^2_\HH
&\leq& \frac{1}{c_\mathcal{G}} \sum_{j=1}^\infty\left|\langle x_{1},g_j\rangle-\sum_{i=1}^\infty \overline{a_{ij}}\langle x_0,g_i\rangle\right|^2 \notag\\
&\leq& \frac{2}{c_\mathcal{G}}
    \sum_{j=1}^\infty\left(|\langle x_{1},g_j\rangle|^2+\left|\sum_{i=1}^\infty \overline{a_{ij}}\langle x_0,g_i\rangle\right|^2\right)\notag \\
&\leq& \frac{2}{c_\mathcal{G}}
    \sum_{j=1}^\infty|\langle x_{1},g_j\rangle|^2+2\frac{C_\mathcal{G}}{c_\mathcal{G}^2} \, \|A\| \sum_{j=1}^\infty |\langle x_0, g_j\rangle|^2 \notag\\
&\leq& \widetilde{C}\, \left(\sum_{j=1}^\infty|\langle x_{1},g_j\rangle|^2+\sum_{j=1}^\infty |\langle x_0, g_j\rangle|^2\right), \label{eq: alg is continuous 2}
\end{eqnarray}
where $\widetilde{C}=2\max\{\frac{1}{c_\mathcal{G}}, \frac{C_\mathcal{G}}{c_\mathcal{G}^2} \, \|A\|_{\HH\to\HH}^2 \} $. These estimations imply the boundedness of the reconstruction formula \eqref{eq: y0 y1}. 

We are now able to define the operator $\mathcal{R}:\mathcal{B}(\ell^2,\C^2)\longrightarrow\HH$ by
\[
\mathcal{R}(D) =  \sum_{j=1}^\infty \left( d_{1j}-\sum_{i=1}^\infty \overline{a_{ij}} \, \left\langle \sum_{k=1}^\infty d_{0k} \, \widetilde{g_k},g_i\right\rangle \right) \widetilde{g}_j,
\]   
where we recall that $[2]=\{0,1\}$ and $D=[d_{nj}]_{n\in [2], \, j\geq 1}$.
Equivalently, by \eqref{eq: coeff aij}, the above operator can be re-written as
\begin{eqnarray}
\mathcal{R}(D) &=&  \sum_{j=1}^\infty \left( d_{1j}- \left\langle \sum_{k=1}^\infty d_{0k} \, \widetilde{g_k},A^*g_j\right\rangle \right) \widetilde{g}_j \notag\\
&=& \sum_{j=1}^\infty \left( d_{1j}- \left\langle A\left(\sum_{k=1}^\infty d_{0k} \, \widetilde{g_k}\right),g_j\right\rangle \right) \widetilde{g}_j.   \label{eq: rec 2}
\end{eqnarray}
By using \eqref{eq: rec 2} and the same bounds as in \eqref{eq: alg is continuous} and \eqref{eq: alg is continuous 2} we  can show that $\mathcal{R}$ is a well-defined bounded operator. Indeed, 
\begin{eqnarray*}
\|\mathcal{R}(D)\|_\HH^2
&=& \left\|\sum_{j=1}^\infty \left( d_{1j}-\left\langle A\left(\sum_{k=1}^\infty d_{0k} \, \widetilde{g_k}\right),g_j\right\rangle \right) \widetilde{g}_j\right\|^2_\HH\\
&\leq&  \widetilde{C}\, \left(\sum_{j=1}^\infty|d_{1j}|^2+\sum_{j=1}^\infty \left|\left\langle \sum_{k=1}^\infty d_{0k} \, \widetilde{g_k}, g_j\right\rangle\right|^2\right)\\
&\leq& \widetilde{C}\, \left(\sum_{j=1}^\infty|d_{1j}|^2+\sum_{j=1}^\infty \left|d_{0j}\right|^2\right)=\widetilde{C}\,\|D\|_{\mathcal{B}(\ell^2,\C^2)}.
\end{eqnarray*}
To justify the last inequality, let $u:= \sum_{k=1}^\infty d_{0k} \, \widetilde{g_k}\in \HH$.  Since $\{g_j\}_{j\ge 1}$ and $\{\widetilde g_j\}_{j\ge 1}$ are a pair of canonical dual frames, we also have that $u= \sum_{j=1}^\infty \langle u,g_j\rangle \, \widetilde{g_j}$. By \cite[Lemma VIII]{DS52}), the coefficients $(\langle u,g_j\rangle)_{j\geq 1}$ have the least $\ell^2$-norm. In particular,
$$\sum_{j=1}^\infty\left|\left\langle \underbrace{\sum_{k=1}^\infty d_{0k} \, \widetilde{g_k}}_{u},g_j\right\rangle\right|^2=\sum_{j=1}^\infty|\langle u,g_j\rangle|^2\leq \sum_{j=1}^\infty|d_{0j}|^2.$$
Finally, the operator $\mathcal{R}$ is linear, and it is a reconstruction operator since by \eqref{eq: y0 y1} we have
\begin{align*}
 \mathcal{R}\left([\langle x_n,g_j\rangle]_{n\in[2], \, j\geq 1}\right) &= \sum_{j=1}^\infty \left( \langle x_1,g_j\rangle-
 \sum_{i=1}^\infty \overline{a_{ij}} \, \left\langle \sum_{k=1}^\infty \underbrace{\langle x_0,g_k\rangle\widetilde{g_k}}_{x_0},g_i\right\rangle \right) \widetilde{g}_j. \\
 &=\sum_{j=1}^\infty \left( \langle x_1,g_j\rangle-\sum_{i=1}^\infty \overline{a_{ij}} \, \langle x_0,g_i\rangle \right) \widetilde{g}_j=w, 
\end{align*}
which concludes the proof. 

\qed

\subsection{Proof of Theorem \ref{thm2}} 
Consider the dynamical system \eqref{model} and suppose that for each $x_0\in\HH$ and $w\in W$ we achieve stable recovery of the source term $w$ in finite time $N$ by measuring the states of the system with a Bessel sequence $\mathcal{G}=\{g_j\}_{j\geq 1}\subset \HH$. Precisely,
the source $w\in W$ can be recovered by applying a bounded linear operator $\mathcal{R}:  \mathcal{B}(\ell^2,\mathbb{C}^N)\rightarrow \HH$. That is, $\mathcal{R}$ satisfies \eqref{bdd} and
\begin{equation}\label{bdd-22}
\mathcal{R}(\mathcal{D}(x_0,w))=w, \qquad  x_0\in \HH,\, w\in W.
\end{equation}
Our objective is to show that, $\{P_W(I-A^*)^{-1}g_j\}_{j\geq 1}$ is a frame for  $W$.

According to \eqref{bdd} and \eqref{bdd-22}, we have
\begin{eqnarray*}
\|w\|_{\HH}^2 
&\leq& C \sum_{n=0}^{N-1}\sum_{j=1}^{\infty} |\langle x_{n}, g_j\rangle|^2\\   
&=& C \left(\sum_{j=1}^{\infty}\sum_{n=0}^{N-1} \left|\left\langle A^nx_0+\sum_{i=0}^{n-1}A^i w, g_j \right\rangle\right|^2 \right)\\
&=& C \left(\sum_{j=1}^{\infty}\sum_{n=0}^{N-1}|\langle A^nx_0+(I-A^n)(I-A)^{-1}w, g_j\rangle|^2 \right), \qquad  w\in W,
\end{eqnarray*}
where we have used that the states of \eqref{model} can be expressed as
\begin{equation}\label{model-22}
x_{n} = A^{n}x_{0}+(I+A+\cdots+A^{n-1})w, \qquad n \geq 1.  
\end{equation}
Given any $w\in W$, by considering $x_0=(I-A)^{-1}w$ and substituting it into the above inequality we get
\begin{eqnarray}\label{eq: Pw and bessel}      
\|w\|_\HH^2
&\leq& NC\sum_{j=1}^{\infty} |\langle (I-A)^{-1}w, g_j\rangle|^2
\leq  C'\|(I-A)^{-1}w\|^2
\leq  C''\|w\|^2, 
\end{eqnarray}
where in the second inequality follows since $\mathcal{G}=\{g_j\}_{j \geq 1}$ is a Bessel sequence in $\HH$, and the last inequality holds because of $(I-A)^{-1}$ is a bounded linear operator on $\HH$. 
Note that, if $C_\mathcal{G}$ denotes the optimal Bessel bound for $\mathcal{G}$, then $C'=NCC_\mathcal{G}$ and $C''=C'\|(I-A)^{-1}\|_{\HH\to\HH}$. 

Since
\begin{equation*}      
\sum_{j=1}^{\infty} |\langle (I-A)^{-1}w, g_j\rangle|^2
= \sum_{j=1}^{\infty} |\langle w, (I-A^*)^{-1}g_j\rangle|^2
= \sum_{j=1}^{\infty} |\langle w, P_W(I-A^*)^{-1}g_j\rangle|^2,
\end{equation*}
from \eqref{eq: Pw and bessel} we have
\begin{equation*}      
\frac{1}{NC}\|w\|_\HH^2
\leq \sum_{j=1}^{\infty} |\langle w, P_W(I-A^*)^{-1}g_j\rangle|^2\leq  C''\|w\|^2, \qquad  w\in W,
\end{equation*}
and so $\{P_W(I-A^*)^{-1}g_j\}_{j\geq 0}$ is a frame for the subspace $W$.
\qed

\section{Proofs of theorems for infinite time iterations} \label{S:proofs-part2}

\begin{lemma}\label{lemma: D in B}
Consider the dynamical system $(\HH, \, W, \mathcal{F}, \, \mathcal{S})$ {(see Definition \ref {genmod})} with any initial state $x_0 \in \HH$. {If $\mathcal{G}=\{g_j\}_{j\geq 1}\subset\HH$ is a Bessel sequence, then} the data matrix {$\mathcal{D}(x_0,w) = [\langle x_{n}, g_j\rangle]_{n\ge0, j\ge1}$} belongs to $\mathcal{B}^s(\ell^2,\ell^\infty)$.
\end{lemma}
\begin{proof}

By assumption, the states $\{x_n\}$ of the dynamical system \eqref{model-2} satisfy
\begin{equation} \label{E:estim-xn-x2}
\|x_n - \mathcal{S}(w)\|_{\HH} \to 0, \qquad n \to \infty.    
\end{equation}
The $n$-th row of $\mathcal{D}(x_0,w) $ is 
\[
r_n := \big(\, \langle x_{n}, g_1\rangle, \, \langle x_{n}, g_2\rangle, \, \dots \,\big).
\]
Define
\[
r := \big(\, \langle \mathcal{S}(w), g_1\rangle, \, \langle \mathcal{S}(w), g_2\rangle, \, \dots \,\big).
\]
Since $\mathcal{G}=\{g_j\}_{j\geq 1}$ is a Bessel sequence in $\HH$ (with optimal Bessel bound $C_\mathcal{G}$), by \eqref{E:estim-xn-x2}, we have
\begin{eqnarray}
\|r_n-r\|_{\ell^2}^{2} 
&=&   \sum_{j=1}^{\infty} \left| \langle x_{n}, g_j\rangle - \langle \mathcal{S}(w), g_j\rangle \right|^2 \notag\\
&=&   \sum_{j=1}^{\infty} \left| \langle (x_{n}-\mathcal{S}(w)), g_j\rangle\right|^2 \notag\\
&\leq&  C_{\mathcal{G}} \|x_{n}-\mathcal{S}(w)\|_{\HH}^2 \to 0, \qquad n\to \infty. \label{E:estim-xi-x2}
\end{eqnarray}
Hence, by Lemma \ref{L:B-strong}, $\mathcal{D}(x_0,w) \in \mathcal{B}^{s}(\ell^2,\ell^\infty)$.    
\end{proof}

\subsection{Proof of Theorem \ref{thmforsubspace-general}} 
Fix any $x_0 \in \HH$ and any source $w\in W$, and put {$
\mathcal{D}(x_0,w) := \big[ \langle x_{i}, g_j\rangle \big]_{n\ge0, j\ge1}$}. By Lemma \ref{lemma: D in B}, $\mathcal{D}(x_0,w) \in \mathcal{B}^{s}(\ell^2,\ell^\infty)$.

Suppose that there is a bounded linear operator $\mathcal{R}: \mathcal{B}^{s}(\ell^2,\ell^\infty) \to \HH$ such that
\[
\mathcal{R} \big( \mathcal{D}(x_0,w) \big) = w \qquad \forall x_0\in\HH, \qquad w\in W.
\]
Hence,
\begin{equation}\label{boundofS2}
\|w\|_\HH \leq \|\mathcal{R}\|_{\mathcal{B}^{s}(\ell^2,\ell^\infty) \to \HH} \, \sup_{n \geq 0} \left(\sum_{j=1}^{\infty} |\langle x_n, g_j\rangle|^2 \right)^{1/2}, \qquad  x_0\in\HH, \, w\in W.
\end{equation}
Since $x_0 \in \HH$ is arbitrary, we can take $x_0=\mathcal{S}(w)$. Hence, $x_n=\mathcal{S}(w)$ for all $n\geq 0$, and \eqref{boundofS2} becomes
\[
\|w\|_{\HH} \leq \|\mathcal{R}\|_{\mathcal{B}^{s}(\ell^2,\ell^\infty) \to \HH} \, \left(\sum_{j=1}^{\infty} |\langle \mathcal{S}(w), g_j\rangle|^2 \right)^{1/2}, \qquad w\in W,
\]
which we rewrite as
\[
m\|w\|_{\HH}^{2} \leq  \sum_{j=1}^{\infty} |\langle w, \, \mathcal{S}^{*} g_j\rangle|^2, \qquad  w\in W
\]
with $m = \|\mathcal{R}\|_{\mathcal{B}^{s}(\ell^2,\ell^\infty) \to \HH}^{-2}$.
 Finally, since $\mathcal{G}$ is a Bessel sequence (with optimal Bessel bound $C_\mathcal{G}$), we also have
\begin{eqnarray*}
\sum_{j=1}^{\infty} |\langle w, \, \mathcal{S}^{*}g_j\rangle|^2 
&=& \sum_{j=1}^{\infty} |\langle \mathcal{S}(w), \, g_j\rangle|^2
\leq C_{\mathcal{G}} \|\mathcal{S}(w)\|_{\HH}^{2}\\
&\leq& C_{\mathcal{G}} \|\mathcal{S}\|_{W \to \HH}^{2} \, \|w\|_{\HH}^{2}= M \, \|w\|_{\HH}^{2}, \qquad \forall w\in W.
\end{eqnarray*}
with $M=C_{\mathcal{G}} \|\mathcal{S}\|_{W \to \HH}^{2}$. Therefore, we conclude that $\{\mathcal{S}^{*}g_j\}_{j\geq 1}$ is a frame for $W.$

For the converse implication, suppose that $\{\mathcal{S}^{*} g_j\}_{j \geq 1}$ is a frame for $W$, and let 
$
\widetilde{\mathcal{G}} = \{\widetilde{g_j}\}_{j \geq 1}\subset W 
$
be a dual frame of $\{\mathcal{S}^{*} g_j\}_{j \geq 1}$. By definition of $\mathcal{D}(x_0,w)$,
\[
(\mathcal{D}(x_0,w) \widetilde{\mathcal{G}})_n = \sum_{j=1}^{\infty} \langle x_n, g_j\rangle \widetilde{g_j}, \qquad n \geq 0. 
\]
Clearly, $(\mathcal{D}(x_0,w) \widetilde{\mathcal{G}})_n \in W$. Since $\{\widetilde{g_j}\}_{j \geq 1}$ is a dual frame of $\{\mathcal{S}^{*}g_j\}_{j \geq 1}$, we can also write
\[
w = \sum_{j=1}^{\infty} \langle w, \mathcal{S}^{*} g_j\rangle\widetilde{g_j}. 
\]
Hence, by \eqref{E:estim-xi-x2}, we have 
\begin{eqnarray*}
\|(\mathcal{D} \widetilde{\mathcal{G}})_n - w\|_{\HH}^2 &=& \left\| \sum_{j=1}^{\infty} \left( \langle x_n, g_j\rangle - \langle w, \mathcal{S}^{*} g_j\rangle\right) \widetilde{g_j} \right\|_{\HH}^2\\
&\leq& C_{\widetilde{\mathcal{G}}} \sum_{j=1}^{\infty} \big| \langle x_n, g_j\rangle - \langle w, \mathcal{S}^{*} g_j\rangle \big|^2\\
&=& C_{\widetilde{\mathcal{G}}} \sum_{j=1}^{\infty} \big| \langle x_n-\mathcal{S}(w), g_j\rangle \big|^2\\
&=& C_{\widetilde{\mathcal{G}}} C_{\mathcal{G}} \|x_{n}-\mathcal{S}(w)\|_{\HH}^2 \to 0.
\end{eqnarray*}
This means that $\mathcal{R}:\mathcal{B}^s(\ell^2,\ell^\infty)\to \HH$ defined by $\mathcal{R}(D)=\lim D\widetilde{\mathcal{G}}$ satisfies
\[
\mathcal{R}\big( \mathcal{D}(x_0,w) \big) = w.
\]
In other words, by Theorem \ref{T:B-strong-recovery},  we have a stable reconstruction.
\qed

\subsection{Proof of Theorem \ref{thmforsubspace}} 
It is enough to show that the conditions of Theorem \ref{thmforsubspace-general} are fulfilled. To do so, note that the equation \eqref{model} simplifies to
\begin{equation} \label{rewritemodel}
x_n=A^nx_0+(I-A^n)(I-A)^{-1}w, \qquad n \geq 1. 
\end{equation}
We rewrite this identity as
\[
x_n - (I-A)^{-1}w = A^n \big( x_0 -(I-A)^{-1}w \big), \qquad n \geq 1. 
\]
Hence,
\[
\|x_n - (I-A)^{-1}w\|_{\HH} =  \|x_0 -(I-A)^{-1}w\|_{\HH} \, \|A^n\|_{\HH \to \HH}, \qquad n \geq 1. 
\]
Since we assumed that $\rho(A)<1$, we have
\[
\|A^n\|_{\HH \to \HH} \to 0, \qquad n \to \infty.
\]
Therefore,
\begin{equation} \label{E:estim-xn-x}
\|x_n - (I-A)^{-1}w\|_{\HH} \to 0, \qquad n \to \infty.    
\end{equation}
This suggests that the stationary point of the dynamical system is uniquely given
\[
\mathcal{S}(w) = (I-A)^{-1}w.
\]
In fact, it is easy to see that if we take $x_0 = (I-A)^{-1}w$, then $x_n=x_0$ for all $n \geq 1$. Moreover, the operator $\mathcal{S} := (I-A)^{-1}|_{W}$ is bounded and invertible. Hence, the conclusion follows from Theorem \ref{thmforsubspace-general}. Note that the adjoint of $\mathcal{S}$ is given by $\mathcal{S}^{*} = P_W(I-A^*)^{-1}$.
\qed

\section{A descriptive Example} \label{S:Example}
We present an example demonstrating that the recovery of $w\in W\subset\HH$ requires an infinite number of time samples when $\HH$ is an infinite-dimensional Hilbert space. It is interesting to note that the subspace $W$ in this example is one-dimensional. 

Let $\HH=\ell^2$, let
\begin{equation} \label{E:def-w-in-l2}
w = (1/2,\, 1/4,\, 1/8,\, \dots)
\end{equation}
with $W$ being the one-dimensional subspace of $\ell^2$ generated by $w$. Let
\[
A =
\begin{bmatrix}
\lambda_1 & 0 & 0 & \cdots\\
0 & \lambda_2 & 0 & \cdots\\
0 & 0 & \lambda_3 & \cdots\\
\vdots & \vdots & \vdots & \ddots\\
\end{bmatrix},
\]
where $0<\lambda_i<1$, $i \geq 1$, and $\lambda_i\ne \lambda_j$ if $i\ne j$. Hence, $A$ acts as a bounded linear operator on $\ell^2$.
    
Let $g=(I-A)w$, which gives $P_W(I-A^*)^{-1}g=w$. Therefore, $\{P_W(I-A^*)^{-1}g\}$ is a frame for $W$. Moreover, by \eqref{E:def-w-in-l2} and that $A$ is diagonal, we have
\begin{equation} \label{E:def-g-in-l2}
g = (g_1,g_2,g_3,\dots) = \big( (1-\lambda_1)/2,\, (1-\lambda_2)/4,\, (1-\lambda_3)/8,\, \dots \big).
\end{equation}
Note that each coordinate $g_i$ is nonzero.

Consider the dynamical system
\[
x_n=Ax_{n-1}+cw,
\] 
for some $c\ne 0$. Then
\begin{equation} \label{E:sustem-1dim}
\langle x_n, g\rangle=\langle x_0,A^n g\rangle+c\langle w,\Lambda_n g\rangle, \qquad n \geq 1,    
\end{equation}
where $\Lambda_n=I+A+\cdots+A^{n-1}$. Let 
\begin{equation} \label{E:def-x-inH}
x_0=(a_1, a_2,\dots, a_N, 0,0,\dots).
\end{equation}
Using \eqref{E:def-w-in-l2} and \eqref{E:def-g-in-l2}, in terms of the coordinates of vectors involved, the system \eqref{E:sustem-1dim} for $n=0,1,\dots,N-1$ can be written as
\[
\begin{bmatrix}
\langle x_0, g\rangle\\
\langle x_1, g\rangle\\
\langle x_2, g\rangle\\
\vdots\\
\langle x_{N-1}, g\rangle\\
\end{bmatrix}
=
\begin{bmatrix}
g_1 & g_2 & \cdots & g_{N} & 0\\
\lambda_1g_1 & \lambda_2g_2 & \cdots & \lambda_Ng_N & b_1\\
\lambda_1g_1 & \lambda_2g_2 & \cdots & \lambda_Ng_N & b_1\\
\vdots & \vdots & \ddots & \vdots & \vdots\\
\lambda_1^{N-1}g_1 & \lambda_2^{N-1}g_2 & \cdots & \lambda_N^{N-1}g_N & b_{N-1}
\end{bmatrix}  
\begin{bmatrix}
a_1\\
a_2\\
a_3\\
\vdots\\
c
\end{bmatrix},
\]
where $b_i=\langle w, \Lambda_i g\rangle$. Note that the matrix is of dimension $N \times (N+1)$. Moreover, since $g_i \ne 0$, $i \geq 1$, we have
\[
\begin{vmatrix}
g_1 & g_2 & \cdots & g_{N}\\
\lambda_1g_1 & \lambda_2g_2 & \cdots & \lambda_Ng_N\\
\lambda_1g_1 & \lambda_2g_2 & \cdots & \lambda_Ng_N\\
\vdots & \vdots & \ddots & \vdots\\
\lambda_1^{N-1}g_1 & \lambda_2^{N-1}g_2 & \cdots & \lambda_N^{N-1}g_N
\end{vmatrix}\ne 0.
\]
Hence, the columns in the above $N \times N$ matrix are linearly independent. Thus, for any $c \ne 0$, there exists a unique $x_0$ of the form \eqref{E:def-x-inH} such that 
\[
\begin{bmatrix}
g_1 & g_2 & \cdots & g_{N} & 0\\
\lambda_1g_1 & \lambda_2g_2 & \cdots & \lambda_Ng_N & b_1\\
\lambda_1g_1 & \lambda_2g_2 & \cdots & \lambda_Ng_N & b_1\\
\vdots & \vdots & \ddots & \vdots & \vdots\\
\lambda_1^{N-1}g_1 & \lambda_2^{N-1}g_2 & \cdots & \lambda_N^{N-1}g_N & b_{N-1}
\end{bmatrix}  
\begin{bmatrix}
a_1\\
a_2\\
a_3\\
\vdots\\
c
\end{bmatrix}
=
\begin{bmatrix}
0\\
0\\
0\\
\vdots\\
0\\
\end{bmatrix}.
\]
Therefore, with this choice of $x_0$, we necessarily have
\[
\langle x_0, g\rangle = \langle x_1, g\rangle = \cdots = \langle x_{N-1}, g\rangle =0,
\]
and thus it is impossible to recover the source term $cw$, if we only have the first $N$ samples of measurements.

\section{Concluding remarks} \label{S:concluding}

We conclude this article by first, underlying the frame condition in the four last main theorems to guarantee stable source recovery, and secondly, by stating two future generalizations of the source recovery problem.

\subsection{Unstable recovery}
In Theorems \ref{thm1}, \ref{thm2}, \ref{thmforsubspace-general}, and \ref{thmforsubspace} the frame condition is a meeting point for stable recovery.
In the following example, we show that reconstruction of the source term is possible under weak assumptions on the sampling set of vectors $\mathcal{G}=\{g_j\}_{j\geq 1}$, but we lack stability.   

Let $\HH=\ell^2$ and consider the dynamical system \eqref{model} with $A=I$ (the identity operator) and measurements given by projecting onto the vectors $g_j:=\frac{1}{j}e_j$ for $j\geq 1$ where $\{e_j\}_{j\geq 1}$ form the standard orthonormal basis for $\ell^2$. In this case, $\mathcal{G}=\{\frac{1}{j}e_j\}_{j\in \N}$ is a Bessel sequence for $\ell^2$ but it is not a frame for $\ell^2$. 
Consider $\mathcal{R}:\mathcal{B}(\ell^2,\C^2)\to\ell^2$ defined by
    $$\mathcal{R}(D)=\sum_{j=1}^\infty j\left(d_{1j}-d_{0j}\right)e_j, \qquad D=[d_{nj}]_{n\in [2], \, j\geq 1}\in \mathcal{B}(\ell^2,\C^2).$$
    Then, given the data matrix $\mathcal{D}(x_0,w)=[\langle x_n,g_j\rangle]_{n\in [2], \, j\geq 1}$, we have 
    $$\mathcal{R}(\mathcal{D}(x_0,w))=\sum_{j=1 }^\infty j\langle x_1-x_0,\frac{1}{j}e_j\rangle e_j=\sum_{j=1}^\infty\langle w,e_j\rangle e_j=w.$$
    In conclusion, although $\mathcal{R}$ is a linear map and allows source reconstruction, it is an unbounded operator.

\subsection{Generalizations of the source recovery problem}

\subsubsection{Time-dependent source term}

If the source term depends on time, that is, if the dynamical system \eqref{model} is replaced by one  of the form
\[
x_{n+1}=Ax_n+w_n, \qquad n \in \N,
\]
where for each $n\in \N$, $w_n\in \HH$, the reconstruction formula provided by Theorem \ref{thm1} can be adopted in this framework. Indeed, assume that the spatial sampling vectors  $\mathcal{G}=\{g_j\}_{j\geq 1}$ form a frame for $\HH$. Then, for recovering each source term $w_n$, two iterations of the dynamical system are needed: Given the measurements $\{\langle x_{n}, g_j\rangle, \langle x_{n+1}, g_j\rangle\}_{j\geq 1}$, the source term $w_n$ can be linearly recovered by 
\begin{gather}\label{eq: recovering formula general n}
w_n= 
\sum_{j=1}^\infty \left(\langle x_{n+1},g_j\rangle-\sum_{i=1}^\infty \overline{a_{ij}}\langle x_n,g_i\rangle\right)\widetilde{g_j}
\end{gather}
where $\{\widetilde{g}_{j}\}_{j\geq 1}$ 
is a dual frame for $\mathcal{G}$ and the coefficients $a_{ij}$ are given according to formula \eqref{eq: coeff aij} (i.e., $a_{ij}=\langle A^*g_j,\widetilde{g}_i\rangle $) 
(cf. \eqref{eq: reconst with 2}).
Notice that if the reconstruction formulas \eqref{eq: recovering formula general n} are used, then we need all the measurements $\mathcal{D}(x_0,w)=[\langle x_{n},g_j\rangle]_{n\geq 0, \, j\geq 1}$ for recovering the source terms $\{w_n\}_{n\in\N}$. The exploration of natural conditions on the measurements $\mathcal{D}(x_0,w)$, as well as necessary and sufficient on $\mathcal{G}$ to guarantee stable recovery in this case, are part of future work.

\subsubsection{Continuous-time dynamical systems} As part of future work, we aim to explore a continuous version of the source recovery problem and address some similarities with the discrete version. Precisely, consider the continuous-time dynamical system on a separable Hilbert space $\HH$
\begin{equation}\label{eq: model cont}
	\begin{cases}
	\dot{x}(t)=Ax(t)+w,\qquad t\in\mathbb R_{\geq 0},\\
	x(0)=x_0
	\end{cases}	
\end{equation} 
where $A$ is the infinitesimal generator of a strongly continuous semigroup $T:\mathbb R_{\geq 0}\to \mathcal{B}(\HH)$, $x_0\in\HH$, and the source term is an unkwown vector $w$ in $\HH$. In the particular case when $A\in \mathcal{B}(\HH)$, then $T(t)=e^{tA}$, where $e^{A}:=\sum\limits_{n=0}^\infty \frac{1}{n!}A^n$. As before, let $\mathcal{G}=\{g_j\}_{j\geq 1}$ be a Bessel sequence for $\HH$, and let us denote by  
    $$\mathcal{D}_c=[\langle x(t), g_j\rangle]$$
the continuous time-space samples of the dynamical system (which depend on $x_0$ and $w$). In this setting, $\mathcal{D}_c$ can be viewed as a curve of observations 
    \begin{gather}
     \mathcal{D}_c:[0,t_0)\to\ell^2 \nonumber\\
     \mathcal{D}_c(t):=(d_1(t),d_2(t),\dots):=(\langle x(t),g_1\rangle, \langle x(t),g_2\rangle, \dots). \label{eq: curve c}
    \end{gather} 
    
    If $\mathcal{G}=\{g_j\}_{j\geq 1}$ is a frame for $\HH$, then the ideas from Theorem \ref{thm1} can be adapted so that one can reconstruct $w$ by considering the continuous time-space samples $\mathcal{D}_c=[\langle x(t), g_j\rangle]_{ t\in [0,t_0), \, j\geq 1 }$ for some $0<t_0\leq \infty$.
    Indeed, by having access to such a curve \eqref{eq: curve c}, we know in particular its value at $t=0$, i.e. 
    $$\mathcal{D}_c(0)=(d_1(0),d_2(0),\dots)=(\langle x_0,g_1\rangle, \langle x_0,g_2\rangle, \dots),$$
    and to its derivative, which using \eqref{eq: model cont} turns out to be
$$\mathcal{D}'(t)=(d'_1(t),d_2'(t),\dots)=(\langle Ax(t), g_1\rangle+\langle w,g_1\rangle, \langle Ax(t), g_2\rangle+\langle w,g_2\rangle, \dots )$$
In particular, 
$$\mathcal{D}'(0)=(d'_1(0),d_2'(0),\dots)=(\langle x_0, A^*g_1\rangle+\langle w,g_1\rangle, \langle x_0, A^*g_2\rangle+\langle w,g_2\rangle,\dots).$$
Therefore, 
$$w= \sum_{j=1}^\infty \left( d_j'(0)-\sum_{i=1}^\infty \overline{a_{ij}}d_i(0) \right) \widetilde{g}_j.$$
where $\{\widetilde{g}_j\}_{j\geq 1}$ is a dual frame for $\mathcal{G}$, and the coefficients $a_{ij}$ are given by \eqref{eq: coeff aij}. Notice that the values $d'_j(0)$ are in place of $d_{1j}=\langle x_1,g_j\rangle$ in \eqref{eq: y0 y1}.

The exploration of natural spaces for the measurements $\mathcal{D}_c$, as well as necessary and sufficient on $\mathcal{G}$ to guarantee stable recovery in this case, are part of future work.

\printbibliography
\end{document}